\documentclass[11pt]{article}
\usepackage{amsmath,amsthm,amssymb}
\usepackage[margin=1in]{geometry}
\usepackage{url}
\usepackage{tikz}
\usepackage{float}
\newcommand{\vD}{\pmb{D}}
\newcommand{\vT}{\pmb{L}}
\newcommand{\vectorp}{\pmb{p}}
\newcommand{\mSigma}{\pmb{\Sigma}}
\newcommand{\vectormu}{\pmb{\mu}}
\newcommand{\event}{\mathcal{E}}
\newcommand{\vX}{\pmb{X}}

\newcommand{\indicator}{\mathbb{I}}
\newcommand{\Prob}{\mathbb{P}}
\newcommand{\convD}{\, \overset{D}{\longrightarrow} \,}
\newcommand{\normal}{{\rm N}}
\newcommand{\F}{\mathbb{F}}
\newcommand{\given}{\, | \,}
\newcommand{\E}{\mathbb{E}}
\newcommand{\Var}{\mathbb{V}{\rm ar}}
\newcommand{\Cov}{\mathbb{C}{\rm ov}}
\newcommand{\Given}{\, \Bigg{|} \,}
\newcommand{\Polya}{P\'{o}lya}

\newtheorem{theorem}{Theorem}
\newtheorem{corollary}{Corollary}
\newtheorem{lemma}{Lemma}
\newtheorem{proposition}{Proposition}

\begin{document}  
\begin{center}
{\huge \bf The degree profile and Gini index of random caterpillar trees}

\bigskip
{\large \bf Panpan Zhang\footnote{Department of Statistics, University of Connecticut, 215 Glenbrook Road U-4120, Storrs, CT 06269-4120, U.S.A. Tel.: +1(860)486-3414. Fax: +1(860)486-4113} \qquad and \qquad
Dipak K.\ Dey{${}^1$}}

\bigskip

\today
\end{center}

\bigskip\noindent
{\bf Abstract:}
In this paper, we investigate the degree profile and Gini index of random caterpillar trees (RCTs). We consider RCTs which evolve in two different manners: uniform and nonuniform. The degrees of the vertices on the central path (i.e., the degree profile) of a uniform RCT follow a multinomial distribution. For nonuniform RCTs, we focus on those growing in the fashion of preferential attachment. We develop methods based on stochastic recurrences to compute the exact expectations and the dispersion matrix of the degree variables. A generalized \Polya\ urn model is exploited to determine the exact joint distribution of these degree variables. We apply the methods from combinatorics to prove that the asymptotic distribution is Dirichlet. In addition, we propose a new type of Gini index to quantitatively distinguish the evolutionary characteristics of the two classes of RCTs. We present the results via several numerical experiments.


\section{Introduction}
\label{Sec:intro}
A {\em caterpillar tree} is a tree in which every vertex has distance at most $1$ from a {\em central path}. The central path of a caterpillar tree is also called the {\em spine} of the tree and it is obtained by removing all {\em endpoint vertices} in the tree. There are different names for endpoint vertices in the literature; for example, terminal vertices, monovalent vertices, leaves, and legs. An explanatory example of a caterpillar tree is given in Figure~\ref{Fig:caterpillar}, in which the central path consists of $5$ vertices, and there are respectively $2$, $1$, $2$, $0$, and $1$ leaves connected to each of them, from left to right.
\begin{figure}[tbh]
	\begin{center}
		\begin{tikzpicture}[scale=3]
		\draw
		(-1.25, 0) node [very thick,circle=0.1,draw] {$p_1$}
		(-0.5, 0) node [very thick,circle=0.1,draw] {$p_2$}
		(0.25, 0) node [very thick,circle=0.1,draw] {$p_3$}
		(1, 0) node [very thick,circle=0.1,draw] {$p_4$}
		(1.75, 0) node [very thick,circle=0.1,draw] {$p_5$};
		\draw[very thick]
		(-1.11, 0)--(-0.64, 0)
		(-0.36, 0)--(0.11, 0)
		(0.39, 0)--(0.86, 0)
		(0.64, 0)--(0.86, 0)
		(1.14, 0)--(1.61, 0);
		\draw
		(-1.5, -0.5) node [circle=0.1,draw] {$l_1$}
		(-1, -0.5) node [circle=0.1,draw] {$l_2$}
		(-0.5, -0.5) node [circle=0.1,draw] {$l_3$}
		(0, -0.5) node [circle=0.1,draw] {$l_4$}
		(0.5, -0.5) node [circle=0.1,draw] {$l_5$}
		(1.75, -0.5) node [circle=0.1,draw] {$l_6$};
		\draw
		(-1.3, -0.13)--(-1.5, -0.36)
		(-1.2, -0.13)--(-1, -0.36)
		(-0.5, -0.14)--(-0.5, -0.36)
		(0.2, -0.13)--(0, -0.36)
		(0.3, -0.13)--(0.5, -0.36)
		(1.75, -0.14)--(1.75, -0.36);
		\end{tikzpicture}
		\caption{An example of a caterpillar tree ($\mathcal{C}(2, 1, 2, 0, 1)$): the central path and the vertices on it are thickened.}
		\label{Fig:caterpillar}
	\end{center} 
\end{figure}
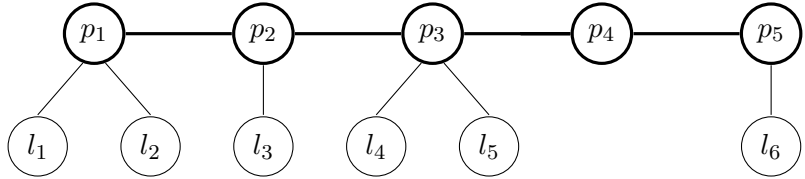

We consider a caterpillar tree with $m \in \mathbb{N}$ (fixed) vertices which are enumerated (from left to right) on the central path. In this paper, the structure of a caterpillar tree is represented by $\mathcal{C}(l_1, l_2, \ldots, l_m)$, where $l_i$ is the number of leaves attached to the node labeled with $i$ on the central path, for $i = 1, 2, \ldots, m$. For example, we denote the caterpillar tree in Figure~\ref{Fig:caterpillar} as $\mathcal{C}(2, 1, 2, 0, 1)$. Caterpillar tree was probably first named by Arthur M.\ Hobbs in~\cite{Hobbs}. Early works on caterpillar trees appeared in combinatorial graph theory~\cite{Harary, Jamison, Miller}. The motivation of this paper originates from the fact that caterpillar trees have found applications in many scientific and applied disciplines. For instance, caterpillar graphs are used to uncover chemical and physical properties of benzenoid hydrocarbons in~\cite{ElBasil}; Caterpillar trees are adopted to model information flow trees in~\cite{Allen}; An polynomial algorithm which determines the total interval number of caterpillar trees is developed in~\cite{Raychaudhuri}; Leaf realization problems for caterpillar trees are investigated in~\cite{Masse}.

Due to the surge of interests in random graphs and algorithms, we incorporate randomness and caterpillar structure, and look into random caterpillar trees (RCTs), which evolve in the following manner. At time $0$, we start with a central path consisting of $m$ vertices, which are enumerated from left to right. At each time point $n \ge 1$, a leaf vertex joins in the tree, and is linked to one of the vertices on the central path via an (undirected) edge according to certain rules, which will be introduced in detail in the sequel. Specifically in this paper, we investigate the degree profile of RCTs and propose a Gini-type index which quantitatively characterizes the evolution of RCTs.

The rest of the manuscript is organized as follows. We study the degree profile of RCTs in Section~\ref{Sec:RCT}, which is divided into two subsections. In Section~\ref{Subsec:URCT}, we look into uniform RCTs, and find that the degree distribution is multinomial. In Section~\ref{Subsec:PARCT}, we place focus on preferential attachment RCTs. We develop stochastic recurrences to compute the first two moments of the degree variables exactly, and exploit a well-known probabilistic model---\Polya\ urn model---to determine the exact degree distribution. In what follows, we show that the asymptotic joint distribution for those degree variables (after properly scaled) is Dirichlet. In Section~\ref{Sec:gini}, we propose a Gini-type index to characterize the evolution of the two classes of RCTs. We show that the proposed Gini index (versus another type of Gini index introduced in~\cite{Balaji}) successfully distinguishes the two classes of RCTs via several simulation studies. Finally, we add some concluding remarks in Section~\ref{Sec:conc} and propose some future work.

\section{Degree profile of random caterpillar trees}
\label{Sec:RCT}
In this section, we investigate the degree profile of RCTs. In graph theory, the {\em degree} of a vertex is the number of edges incident to the vertex. For each of the vertices on the central path of a RCT, its degree is composited by two parts: the number of adjoint leaves and the number of links on the central path. The former is random, while the latter is deterministic: $1$ for the vertices at the two endpoints of the spine, and $2$ for the rest. Let $\vD_n = (D_{1, n}, D_{2, n}, \ldots, D_{m, n})^{\top}$ be the random vector that represents the degree profile of the $m$ vertices on the central path of a RCT at time $n$. The structure of a RCT is analogously represented by $\mathcal{C}(L_{1, n}, L_{2, n}, \ldots, L_{m, n})$, where $L_{1, n}, L_{2, n}, \ldots, L_{m, n}$ are respectively the random variables which refer to the number of leaves attached to each of the $m$ vertices on the spine at time $n$. The composition of these random variables is denoted by a random vector $\vT_n = (L_{1, n}, l_{2, n}, \ldots, L_{m, n})^{\top}$. There is an instant relationship between $\vD_n$ and $\vT_n$ for all $n \ge 0$:
\begin{equation}
\vD_n = \vT_n + (1, 2, 2, \ldots, 2, 2, 1)^{\top}. \label{Eq:relation}
\end{equation}

We consider two types of RCTs: uniform random caterpillar trees ($\mathcal{C}^{(U)}(L_{1, n}, L_{2, n}, \ldots, L_{m, n})$) and preferential attachment random caterpillar trees ($\mathcal{C}^{(P)}(L_{1, n}, L_{2, n}, \ldots, L_{m, n})$), which are distinguished by the features of their growth. The evolution of uniform RCTs is analogous to that of random recursive trees. At each time point $n \ge 1$, a vertex on the central path is chosen uniformly at random (all vertices being equally likely) and connected with a newcomer via an edge. Preferential attachment RCTs grow in a nonuniform way, inspired from the seminal paper~\cite{Barabasi}. At time $n \ge 1$, the probability of a vertex on the central path being selected for a newcomer is proportional to its degree in the tree at time $n - 1$. Mathematically, given $\mathcal{C}^{(P)}(L_{1, n - 1}, L_{2, n - 1}, \ldots, L_{m, n - 1})$, the probability that the vertex labeled with $i$, for $i = 1, 2, \ldots, m$, on the central path is chosen for the new leaf vertex (newcomer) at time $n$ is
$$\frac{D_{i, n - 1}}{\sum_{i = 1}^{m} D_{i, n - 1}} = \frac{L_{i, n - 1} + \indicator_{(\{i = 1\} \cup \{i = m\})} + 2 \, \indicator_{\{2 \le i \le m - 1\}}}{\sum_{i = 1}^{n} L_{i, n - 1} + 2m - 2},$$
where $\indicator_{(\cdot)}$ denotes an indicator function.

\subsection{Uniform random caterpillar trees}
\label{Subsec:URCT}
The degree profile of uniform RCTs is trivial, recovered by some well-known results from fundamental probability theory. The growth of uniform RCTs coincides with an experiment of $n$ independent trials, each of which leads to a choice for one of $m$ candidates, with every candidate having a fixed success rate $1/m$. The associated distribution is a {\em multinomial distribution} with parameters $n$ and $\vectorp = (1/m, 1/m, \ldots, 1/m)$. We thus obtain the joint probability mass function of $L_{1, n}, L_{2, n}, \ldots, L_{m, n}$:
$$\Prob(L_{1, n} = l_1, L_{2, n} = l_2, \ldots, L_{m, n} = l_m) = \binom{n}{l_1, l_2, \cdots, l_m} \left(\frac{1}{m}\right)^n,$$
for nonnegative integers $l_1, l_2, \cdots, l_m$, and $\sum_{i = 1}^{m} l_i = n$. According to the relation between $\vD_n$ and $\vT_n$ in Equation~(\ref{Eq:relation}), we get the joint distribution of $D_{1, n}, D_{2, n}, \ldots, D_{m, n}$; namely,	$$\Prob(D_{1, n} = d_1, D_{2, n} = d_2, \ldots, D_{m, n} = d_m) = \binom{n + 2m - 2}{d_1, d_2, \cdots, d_m} \left(\frac{1}{m}\right)^{n + 2m - 2},$$
for integers $d_1 \ge 1, d_m \ge 1$, $d_2, d_3, \ldots, d_{m - 1} \ge 2$ and $\sum_{i = 1}^{m} d_i = n + 2m - 2$.

Several limiting distributions of linear functions of multinomial distributed random variables are given in~\cite{Fisz}. Following the results in~\cite[Theorem 1]{Fisz}, we obtain the limiting distribution of $\vT_n$. As $n \to \infty$, we have
$$\frac{\vT_n - n \vectorp}{\sqrt{n}} \convD \normal_{m}({\bf 0}, \mSigma^{(U)}),$$
where $\vectorp = (1/m, 1/m, \ldots, 1/m)^{\top}$, and the dispersion matrix $\mSigma^{(U)}$ is
$$\begin{pmatrix}
(m - 1)/m^2 & -1/m^2 &\cdots & -1/m^2
\\ -1/m^2 & (m - 1)/m^2 &\cdots & -1/m^2
\\ \vdots & \vdots & \ddots & \vdots
\\ -1/m^2 & -1/m^2 & \cdots & (m - 1)/m^2
\end{pmatrix}.$$
Accordingly, the asymptotic distribution $\vD_n$ is normal after properly scaled; that is,
$$\frac{\vD_n - n \vectorp}{\sqrt{n}} \convD \normal_{m}(\vectormu^{(U)}, \mSigma^{(U)}),$$
where $\vectormu^{(U)} = (1, 2, 2, \ldots, 2, 2, 1)^{\top}$.

\subsection{Preferential attachment random caterpillar trees}
\label{Subsec:PARCT}
In contrast to uniform RCTs, preferential attachment RCTs evolve in a flavor of the vertices with higher degrees being more attractive to newcomers. The first consideration of preferential attachment seems to appear in~\cite{Yule}, and one of the most broad applications of preferential attachment is to model the growth of the World Wide Web in~\cite{Barabasi}. In sociology, the phenomenon of preferential attachment is reflected in a well known manifestation: ``the rich get richer and the poor get poorer.'' 

The recruiting candidates for newcomers in uniform RCTs are chosen independently from time to time, while the recruiting process in preferential attachment RCTs at each time point is dependent on the structure of the existing tree at the preceding time point. The strong dependency between the trees at two consecutive time points under the preferential attachment setting makes computation much more challenging.

In this section, we first compute the degree vector $\vD_n$'s first two moments, which would provide us an insight into the distribution of $\vD_n$. Let $\F_n$ be the $\sigma$-field that contains the history of the evolution of a preferential attachment RCT up to time $n$ (i.e., $\F_n$ is the $\sigma$-field generated by $\vD_0, \vD_1, \ldots, \vD_n$). At $n = 0$, the initial condition is
\begin{equation}
\vD_0 = (1, 2, 2, \ldots, 2, 2, 1)^{\top}.
\label{Eq:initial}
\end{equation}
\begin{proposition}
	\label{Prop:mom}
	Let $\mathcal{C}^{(P)}(L_{1, n}, L_{2, n}, \ldots, L_{m, n})$ be a preferential attachment RCT at time $n$, and let $\vD_n = (D_{1, n}, D_{2, n}, \ldots, D_{m, n})^{\top}$ be the random vector that represents the degree profile of the $m$ vertices on the central path. The expectation of $\vD_n$ is
	$$\E[\vD_{n}] = \begin{pmatrix}
	\frac{n}{2(m - 1)} + 1
	\\ \frac{n}{m - 1} + 2
	\\ \vdots
	\\ \frac{n}{m - 1} + 2
	\\ \frac{n}{2(m - 1)} + 1
	\end{pmatrix}.$$
	The dispersion matrix of $\vD_n$, denoted by $\mSigma_n = \bigl(\sigma(n)_{i, j}\bigr)_{i, j = 1}^{m}$, is an $m \times m$ square matrix such that
	$$
	\sigma(n)_{i, i} = \begin{cases}
	\frac{(m - 2)n^2 + 2(m - 2)(m - 1)n}{(m - 1)^2(2m - 1)}, &\qquad 2 \le i \le m - 1,
	\\ \frac{(2m - 3)n^2 + (2m - 3)(2m - 2)n}{4(m - 1)^2(2m - 1)}, &\qquad {i = 1, m}.
	\end{cases}
	$$
	and
	$$
	\sigma(n)_{i, j} = \begin{cases}
	-\frac{n^2 + 2(m - 1)n}{(m - 1)^2(2m - 1)}, &\qquad 2 \le i \neq j \le m - 1,
	\\ -\frac{n^2 + 2(m - 1)n}{2(m - 1)^2(2m - 1)}, &\qquad 
	{{i = 1, m \mbox{ and } 2 \le j \le m - 1} \atop {j = 1, m \mbox{ and } 2 \le i \le m - 1}},
	\\ -\frac{n^2 + 2(m - 1)n}{4(m - 1)^2(2m - 1)}, &\qquad {{i = 1 \mbox{ and } j = m} \atop {j = 1 \mbox{ and } i = m}}.
	\end{cases}
	$$
\end{proposition}
\begin{proof}
	We look into each of the components in $\vD_n$. For each $1 \le i \le m$ and $n \ge 1$, there is an almost-sure relation between $D_{i, n}$ and $D_{i, n - 1}$:
	\begin{equation}
	D_{i, n} \given \F_{n - 1} = (D_{i, n - 1} + \indicator_{\event_{i, n}}) \given \F_{n - 1},
	\label{Eq:almostrelation}
	\end{equation}
	where $\event_{i, n}$ indicates the event of the vertex labeled with $i$ on the central path being selected for the newcomer at time $n$. Taking expectations on both sides of Equation~(\ref{Eq:almostrelation}), we get
	$$\E[D_{i, n} \given \F_{n - 1}] = D_{i, n - 1} + \frac{D_{i, n - 1}}{2m - 2 + n - 1}.$$
	Taking another expectation, we obtain a recurrence relation for $\E[D_{i, n}]$ with respect to $n$, i.e.,
	$$\E[D_{i, n}] = \frac{n + 2m - 2}{n + 2m - 3} \, \E[D_{i, n - 1}].$$
	Solving the recurrence relation with the initial condition given in Equation~(\ref{Eq:initial}), we obtain the result for the first moment as stated in the proposition.
	
	Towards the dispersion matrix of $\vD_n$, we again appeal to the stochastic relation established in Equation~(\ref{Eq:almostrelation}) to compute the second moments of $D_{i, n}$'s for $1 \le i \le m$ and the mixed moments of $D_{i, n}$ and $D_{j, n}$ for $0 \le i \neq j \le m$.
	
	For each fixed $1 \le i \le m$, we square both sides of Equation~(\ref{Eq:almostrelation}) to get
	\begin{equation}
	D^2_{i, n} \given \F_{n - 1} = D^2_{i, n - 1} + 2D_{i, n - 1}\indicator_{\event_{i, n}} \given \F_{n - 1} + \indicator_{\event_{i, n}} \given \F_{n - 1}.
	\label{Eq:square}
	\end{equation}
	The recurrence for $\E[D^2_{i, n}]$ is obtained by taking expectations on both sides of Equation~(\ref{Eq:square}) twice and plugging in the expectation of $D_{i, n}$,
	\begin{align*}
	\E[D^2_{i, n}] &= \frac{n + 2m - 1}{n + 2m - 3} \, \E[D^2_{i, n - 1}] + \frac{n + 2m - 2}{(m - 1)(n + 2m - 3)}, & \mbox{for }2 \le i \le m,
	\\ \E[D^2_{i, n}] &= \frac{n + 2m - 1}{n + 2m - 3} \, \E[D^2_{i, n - 1}] + \frac{n + 2m - 2}{2(m - 1)(n + 2m - 3)}, & \mbox{for }i = 1, m.
	\end{align*}
	Solving the stochastic recurrences with the initial condition (cf.\ Equation~(\ref{Eq:initial})), we get
	\begin{align*}
	\E[D^2_{i, n}] &= \frac{3n^2 + 2(5m - 4)n + 2(2m - 2)(2m - 1)}{(2m - 1)(m - 1)}, \quad & \mbox{for }2 \le i \le m,
	\\ \E[D^2_{i, n}] &= \frac{2n^2 + (6m - 5)n + (2m - 2)(2m - 1)}{2(2m - 1)(m - 1)}, \quad & \mbox{for }i = 1, m.
	\end{align*}
	Accordingly, we obtain the variances for $D_{i, n}$'s, which form the diagonal of the variance and covariance matrix of $\vD_n$:
	\begin{align*}
	\sigma(n)_{i, i} &= \Var[D_{i, n}] = \frac{(m - 2)n^2 + 2(m - 2)(m - 1)n}{(m - 1)^2(2m - 1)}, \quad & \mbox{for }2 \le i \le m,
	\\ \sigma(n)_{i, i} &= \Var[D_{i, n}] = \frac{(2m - 3)n^2 + (2m - 3)(2m - 2)n}{4(m - 1)^2(2m - 1)}, \quad & \mbox{for }i = 1, m.
	\end{align*}
	To compute the covariances between $D_{i, n}$ and $D_{j, n}$ for $1 \le i \neq j \le m$, we need the mixed moments of $D_{i, n}$ and $D_{j, n}$, i.e., $\E[D_{i, n}D_{j, n}]$. Recall the almost-sure relation between $D_{i, n}$ and $D_{i, n - 1}$ in Equation~(\ref{Eq:almostrelation}). For $i \neq j$, we have
	\begin{align}
	D_{i, n} D_{j, n} \given \F_{n - 1} &= \bigl((D_{i, n - 1} + \indicator_{\event_{i, n}})(D_{j, n - 1} + \indicator_{\event_{j, n}}) \bigr)\given \F_{n - 1} \nonumber
	\\ &= (D_{i, n - 1}D_{j, n - 1} + D_{j, n - 1}\indicator_{\event_{i, n}} + D_{i, n - 1}\indicator_{\event_{j, n}}) \given \F_{n - 1}.
	\label{Eq:mixed}
	\end{align}
	In Equation~(\ref{Eq:mixed}), the term $\indicator_{\event_{i, n}}\indicator_{\event_{j, n}} \given \F_{n - 1}$ vanishes as the events $\event_{i, n}$ and $\event_{j, n}$ are mutually exclusive. In what follows, we obtain a recurrence for $\E[D_{i, n} D_{j, n}]$; that is,
	$$\E[D_{i, n} D_{j, n}] = \frac{n + 2m - 1}{n + 2m - 3} \E[D_{i, n - 1} D_{j, n - 1}].$$
	Solving the equation above recursively, we get
	$$\E[D_{i, n}D_{j, n}] = \frac{2\bigl(n^2 + (4m - 3)n + (2m - 2)(2m - 1)\bigr)}{(2m - 1)(m - 1)},$$
	for $2 \le i \neq j \le m - 1$;
	$$\E[D_{i, n}D_{j, n}] = \frac{n^2 + (4m - 3)n + (2m - 2)(2m - 1)}{(2m - 1)(m - 1)},$$
	for $i = 1, m$ and $2 \le j \le m - 1$ (or vice versa); 
	$$\E[D_{i, n}D_{j, n}] = \frac{n^2 + (4m - 3)n + (2m - 2)(2m - 1)}{2(2m - 1)(m - 1)},$$
	for $ i = 1$ and $j = m$ (or vice versa). Thus, we obtain other entries in the variance and covariance matrix of $\vD_n$. For $2 \le i \neq j \le m - 1$, we have
	$$\sigma(n)_{i, j} = \Cov(D_{i, n}, D_{j, n}) = -\frac{n^2 + 2(m - 1)n}{(m - 1)^2(2m - 1)};$$
	for $i = 1, m$ and $2 \le j \le m - 1$ (or vice versa), we have
	$$\sigma(n)_{i, j} = \Cov(D_{i, n}, D_{j, n}) = -\frac{n^2 + 2(m - 1)n}{2(m - 1)^2(2m - 1)};$$
	and for $ i = 1$ and $j = m$ (or vice versa), we have
	$$\sigma(n)_{i, j} = \Cov(D_{i, n}, D_{j, n}) = -\frac{n^2 + 2(m - 1)n}{4(m - 1)^2(2m - 1)}.$$
\end{proof}

Next, we look at the asymptotic distribution of $\vD_n$ for large $n$. We exploit a {\em Martingale Convergence Theorem} to prove that the limiting distribution of $\vD_n$ (after properly scaled) exists. We first give some quick words about {\em martingale}. Martingale is a popular and powerful mathematical tool owing to its conceptual simplicity and versatility. A general definition of martingale can be found in~\cite[Section 1.1]{Hall}, which will be omitted here. Martingale has found applications in various research areas: theoretical probability theory~\cite{Hall}, applied probability~\cite{Mahmoud}, stochastic processes~\cite{Doob} and financial modeling~\cite{Musiela}. 

The $\sigma$-field sequence $\{\F_n\}_n$ defined in Subsection~\ref{Subsec:PARCT} forms a filtration in our martingale setting. However, for each fixed $1 \le i \le m$, the random variables $D_{i, n}$ do not form a martingale sequence (with respect to $\{\F_n\}_n$). We introduce a transformation to $D_{i, n}$ in the next lemma, and the new sequence is a martingale.

\begin{lemma}
	\label{Lem:martingale}
	For each $1 \le i \le m$, the random variables
	$$M_{i, n} = \frac{2(m - 1)}{n + 2m - 2} \, D_{i, n}$$ 
	form a martingale sequence.
\end{lemma}
\begin{proof}
	We check the requirements for martingales one after another. At first, it is obvious that $M_{i, n}$ is measurable on $\F_n$ as $D_{i, n}$ is measurable on $\F_{n}$. In addition, we observe that $\E[D_{n, i}] = O(n)$ for all $1 \le i \le m$ according to Proposition~\ref{Prop:mom}. Finally, by Equation~(\ref{Eq:almostrelation}), we have
	$$\E[M_{i, n} \given \F_{n - 1}] = \E\left[\frac{2(m - 1)}{n + 2m - 2} D_{i, n} \Given \F_{n - 1} \right]
	= \frac{2(m - 1) D_{i, n - 1}}{n + 2m - 3} = M_{i, n - 1}.$$
	which completes the verification.
\end{proof}

\begin{theorem}
	\label{Thm:convergence}
	As $n \to \infty$, there exists a random vector $\vX$ such that
	$$\frac{\vD_n}{n} \convD \vX.$$
\end{theorem}
\begin{proof}
	For each fixed $i$, recall the martingale sequence $\{M_{i, n}\}_n$ established in Lemma~\ref{Lem:martingale}. By the construct of $M_{i, n}$ and Proposition~\ref{Prop:mom}, it is obvious that $M_{i, n}$ is $\mathcal{L}^1$-bounded. According to the Martingale Convergence Theorem~\cite[Theorem 2.5]{Hall}, we conclude that there exits a random variable~$X_i$, to which $M_{i, n}$ converges almost surely, as $n \to \infty$. For each $i$, we set $\tilde{X}_i = X_i/(2(m - 1))$. The random vector $\vX = (\tilde{X}_1, \tilde{X}_2, \ldots, \tilde{X_n})^{\top}$ is the limit as stated in the theorem.
\end{proof}

We prove the existence of the limiting distribution of $\vD_n$ in Theorem~\ref{Thm:convergence}. However, the limiting distribution is not determined. Next, we introduce a probabilistic model---{\em \Polya\ urn model}---to characterize the dynamics of the degree variables, and thus find the exact distribution of $\vD_n$, followed by the limiting distribution. We refer the interested readers to~\cite{Mahmoud} for the history, definition, and applications of \Polya\ urn models. In this paper, we focus on a \Polya\ urn generalized from a classical model---the {\em \Polya-Eggenberger urn}~\cite{Eggenberger}.

Consider an urn containing $k$ different types of balls (e.g., $k$ different colors). Initially, the urn contains a total number of $\tau_{0}$ balls, of which there are $\tau_{i, 0}$ balls of color $i$, for $i = 1, 2, \ldots, m$, and $\tau_0 = \sum_{i = 1}^{m} \tau_{i, 0}$. At each time point $n \ge 1$, a ball is chosen from the urn uniformly at random, its color is observed, and the ball is placed back to the urn in addition with a ball of the same color. The dynamics of the urn scheme is governed by an $m \times m$ {\em replacement matrix}:
$$\begin{pmatrix}
1 & 0 & \cdots & 0
\\ 0 & 1 & \cdots & 0
\\ \vdots & \vdots & \ddots & \vdots
\\ 0 & 0 & \cdots & 1
\end{pmatrix},$$
where the rows are indexed with colors $1, 2, \ldots, m$ from top to bottom, and the columns are indexed with colors $1, 2, \ldots, m$ from left to right. The dynamic of the degree addition in a preferential attachment RCT is associated with an $m$-color \Polya-Eggenberger urn with the initial condition： $\tau_{1, 0} = 1, \tau_{2, 0} = 2, \ldots, \tau_{m - 1, 0} = 2, \tau_{m, 0} = 1$.

A remarkable property of the \Polya-Eggenberger urns is {\em exchangeability}, i.e., the probabilities of choosing balls of different colors in all $n$-long sequences which have the same number of balls sampled for each color are identical, not depending on the order of those balls chosen in the sequence. The exact joint distribution of degree variables $D_{1, n}, D_{2, n}, \ldots, D_{m, n}$ is given in the next theorem.
\begin{theorem}
	\label{Thm:exactdist}
	Let $\mathcal{C}^{(P)}(L_{1, n}, L_{2, n}, \ldots, L_{m, n})$ be a preferential attachment RCT at time $n$, and let $\vD_n = (D_{1, n}, D_{2, n}, \ldots, D_{m, n})^{\top}$ be the random vector that represents the degree profile of the $m$ vertices on the central path. Suppose that the balls of color $i$ are chosen $s_i$ times in the $n$-long sampling sequence, we have
	$$ \Prob(D_{1, n} = \tau_{1, 0} + s_1, \ldots, D_{m, n} = \tau_{m, 0} + s_{m}) = \binom{n}{s_1, s_2, \cdots, s_m} \frac{\prod_{i = 1}^{m} \langle \tau_{i, 0} \rangle_{s_i}}{\langle \tau_0 \rangle_{n}},$$
	where $(\tau_{1, 0}, \tau_{2, 0}, \ldots, \tau_{m - 1, 0}, \tau_{m, 0})^{\top} = (1, 2, \ldots, 2, 1)^{\top}$, $0 \le s_1, s_2, \ldots, s_m \le n$, $\sum_{i = 1}^{m} s_i = n$, $\tau_0 = \sum_{i = 1}^{m} \tau_{i, 0} = 2m - 2$, and $\langle \cdot \rangle$ refers to the Pochhammer symbol of the rising factorial.
\end{theorem}
\begin{proof}
	Consider an $m$-color \Polya-Eggenberger urn starting with $\tau_{i, 0}$ balls of color $i$, for $i = 1, 2, \ldots, m$. A possible string (sequence) to obtain an urn containing $\tau_{i, 0} + s_i$ balls of color $i$ is to sample balls of color $1$ in the first $s_1$ steps of the $n$-long sequence, sample balls of color $2$ in the next $s_2$ steps, and continue sampling in this manner until the balls of color $m$ are selected in the last $s_m$ steps in the sequence. The probability of obtaining this particular sampling string is
	\begin{align*}
	&\frac{\tau_{1, 0}(\tau_{1, 0} + 1)\cdots(\tau_{1, 0} + s_1 - 1)}{\tau_0(\tau_0 + 1)\cdots(\tau_{0} + s_1 - 1)} \times \frac{\tau_{2, 0}(\tau_{2, 0} + 1)\cdots(\tau_{2, 0} + s_2 - 1)}{(\tau_0 + s_1)(\tau_0 + s_1 + 1)\cdots(\tau_{0} + s_1 + s_2 - 1)}
	\\&\qquad{}\times \cdots \times \frac{\tau_{m, 0}(\tau_{m, 0} + 1)\cdots(\tau_{m, 0} + s_m - 1)}{(\tau_0 + s_1 + \cdots + s_{m - 1})\cdots(\tau_{0} + s_1 + \cdots + s_{m} - 1)}.
	\end{align*}
	There is a total of $\binom{n}{s_1, s_2, \cdots, s_m}$ strings to achieve the outcome of the urn containing $\tau_{i, 0} + s_i$ balls of color $i$ at time $n$. By the property of exchangeability, we obtain the stated joint probability mass function of $D_{i, n}$'s.
\end{proof}

\begin{corollary}
	\label{Cor:asymdist}
	As $n \to \infty$, we have
	$$\left(\frac{D_{1, n}}{n}, \frac{D_{2, n}}{n}, \ldots, \frac{D_{m, n}}{n}\right) \convD {\rm Dir}(\tau_{1, 0}, \tau_{2, 0}, \ldots, \tau_{m, 0}),$$
	where $\tau_{1, 0} = 1, \tau_{2, 0} = 2, \ldots, \tau_{m, 0} = 1$ are the parameters of ($m$-dimensional) Dirichlet distribution.
\end{corollary}
\begin{proof}
	We write the joint probability mass function of $D_{1, n}, D_{2, n}, \ldots, D_{m, n}$ presented in Theorem~\ref{Thm:exactdist} in terms of gamma functions; that is,
	\begin{align}
	\Prob(D_{1, n} = \tau_{1, 0} + s_1, \ldots, D_{m, n} = \tau_{m, 0} + s_{m}) &= \frac{\Gamma(n + 1)}{\Gamma(s_1 + 1) \Gamma(s_2 + 1) \cdots \Gamma(s_m + 1)} \nonumber
	\\ &\qquad{}\times \frac{\prod_{i = 1}^{m} \bigl(\Gamma(\tau_{i, 0} + s_i) / \Gamma(\tau_{i, 0}) \bigr)}{\Gamma(\tau_0 + n) / \Gamma(\tau_0)} \nonumber
	\\ &= \frac{\Gamma(\tau_0)}{\prod_{i = 1}^{m} \Gamma(\tau_{i, 0})} \frac{\Gamma(n + 1)}{\Gamma(\tau_0 + n)} \nonumber
	\\ &\qquad{}\times \frac{\prod_{i = 1}^{m}\Gamma(\tau_{i, 0} + s_i)}{\prod_{i = 1}^{m} \Gamma(s_i + 1)}. \label{Eq:exactdist}
	\end{align}
	Noting that $0 \le s_1, s_2, \ldots, s_m \le n$ and $s_1 + s_2 + \ldots + s_m = n$, we define $\theta_i = s_i / n$, and find that the support of $\theta_i$'s is $0 \le \theta_1, \theta_2, \ldots, \theta_m \le 1$ such that $\sum_{i = 1}^{m} \theta_i = 1$. Replace $s_i$ in Equation~(\ref{Eq:exactdist}) by $n\theta_{i}$ for each $1 \le i \le m$. We then apply the {\em Stirling's approximation}~\cite[Equation (4.23)]{Graham} to the ratio of gamma functions in Equation~(\ref{Eq:exactdist}) as $n \to \infty$, and conclude that 
	$$\left(\frac{D_{1, n}}{n}, \frac{D_{2, n}}{n}, \ldots, \frac{D_{m, n}}{n}\right) \convD \frac{\Gamma(\tau_0)}{\prod_{i = 1}^{m} \Gamma(\tau_{i, 0})} \prod_{i = 1}^{m} (\theta_i)^{\tau_{i, 0} - 1},$$
	for $\sum_{i = 1}^{m} \theta_i = 1$, which is the probability density function of a Dirichlet distribution with parameters $(\tau_{1, 0}, \tau_{2, 0}, \ldots, \tau_{m, 0})$.
\end{proof}

\section{Gini index of random caterpillar trees}
\label{Sec:gini}

In this section, we propose a Gini-type index to characterize the evolution of the two classes of RCTs considered in Section~\ref{Sec:RCT}. The {\em Gini index}, named after the Italian statistician and sociologist Corrado Gini, arose from a problem of measuring statistical dispersion of wealth distribution of national residents in economics. In modern times, the Gini index is extended to a commonly-used measure of inequality of a distribution, which has found applications in medicine~\cite{Lee}, public health~\cite{Kennedy}, physics~\cite{Roe}, chemistry~\cite{Graczyk} and complex networks~\cite{Hu}, etc. Statisticians are committed to establishing and developing rigorous methods to calculate or estimate the Gini index, see representative papers such as~\cite{Gastwirth, Lerman, Ogwang}. Very recently, the Gini index was exploited to measure the sparsity of a network~\cite{Goswami}. One of the most effective ways to illustrate the Gini index may be to exploit a graphical representation---the {\em Lorenz curve}; see~\cite{Gastwirth}. In this paper, we propose a Gini-type index which quantifies the disparity within different classes of RCTs so as to characterize their evolution. We also compare the proposed Gini index with the one recently introduced in~\cite{Balaji} via some numerical experiments. To distinguish the two Gini indicies in the rest of the manuscript, we call the Gini index from~\cite{Balaji} type {\rm I} Gini index, and the one proposed herein type {\rm II} Gini index.

\subsection{Type {\rm I} Gini index}
\label{Subsec:type1}
The first type of Gini index (i.e., type {\rm I} Gini index) that we look into is proposed in~\cite{Balaji}, the authors of which considered a Gini-type topological index for several classes of random rooted trees. In particular, they illustrate the estimation of their measure via a class of uniform RCTs. 

To begin with, we give a quick word about type {\rm I} Gini index. Let $\mathcal{T}$ be a class of rooted trees. For each vertex $v$ in an arbitrary tree $T \in \mathcal{T}$, the {\em geodesic distance} between $v$ and the root (this measure is sometimes expressed as the {\em depth} of $v$) is the number of edges in a shortest path connecting them, denoted by $d_v$. If we consider all the vertices in $T$ as our target population, and the ``wealth'' of each of them is represented by $d_v$, then the associated Gini index, i.e., type {\rm I} Gini index of $T$, is given by
\begin{equation}
\label{Eq:gini1ext}
G_{\rm I}(T) = \frac{\sum_u \sum_v |d_u - d_v|}{2 \sum_u \sum_v d_v},
\end{equation}
for all $u, v \in T$. The estimator of type {\rm I} Gini index for an arbitrary class of rooted trees $\mathcal{T}$, denoted by $\hat{G}_{\rm I}(\mathcal{T})$, is developed on the procedures introduced by~\cite{Gastwirth}. Let $\{\# \mathcal{T}\}$ be the cardinality of $\mathcal{T}$. The estimator of $G_{\rm I}(\mathcal{T})$ is given by
\begin{equation}
\hat{G}_{\rm I}(\mathcal{T}) = \frac{\sum_u \sum_v \E|d_u - d_v|}{2 \left(\E[{\# \mathcal{T}}]\right)^2 \E[d_{v}]},
\label{Eq:gini1gen}
\end{equation}
for all $u, v \in T$ and an arbitrary $T \in \mathcal{T}$.

According to the estimator in Equation~(\ref{Eq:gini1gen}), we calculate type {\rm I} Gini indices of the classes of uniform RCTs and preferential attachment RCTs, respectively. Without loss of generality, we consider the vertex at the leftmost position on the central path as the root. For better readability, we only present the results in the main body of the paper, but more algebra can be found in the appendix (Section~\ref{Subsec:gini1u}). The type {\rm I} Gini index of uniform RCTs at time $n$ is given by
$$\hat{G}_{\rm I}\left(\mathcal{C}^{(U)}\right) = \frac{(m - 1)\left[(m + 1)n^2 + (m^2 + 3m - 1)n + m^3 + m^2\right]}{3m (n + m) \left[(m + 1)n + m^2 - 2m\right]}.$$
As $n$ goes to infinity, we see that $\hat{G}_{\rm I}\left(\mathcal{C}^{(U)}\right)$ converges to $(m - 1)/(3m)$. For a large value of $m$, this index is close to $1/3$, which is consistent with the conclusion drawn in~\cite{Balaji}.

We verify our conclusion via a Monte-Carlo experiment, the graphical result of which is depicted in Figure~\ref{Fig:ginioneunifRCT}. In the experiment, we simulate $40$ classes of uniform RCTs at time $n = 500$ according to $40$ different values of $m$: $5, 10, \ldots, 200$. For each class of uniform RCTs, the replication number $R$ is set at $500$. Note that the size of an arbitrary RCT from any class is deterministic; that is, $m +　n$. For each simulated uniform RCT, we determine the depth of each vertex in our simulations, and compute type {\rm I} Gini index via the formula in Equation~(\ref{Eq:gini1ext}). The estimate of this type of Gini index (for each class) is obtained by averaging over all type {\rm I} Gini indices of the $500$ replications. 

\begin{figure}[ht]
	\begin{center}
		\begin{minipage}[b]{0.45\textwidth}
			\includegraphics[width=\textwidth]{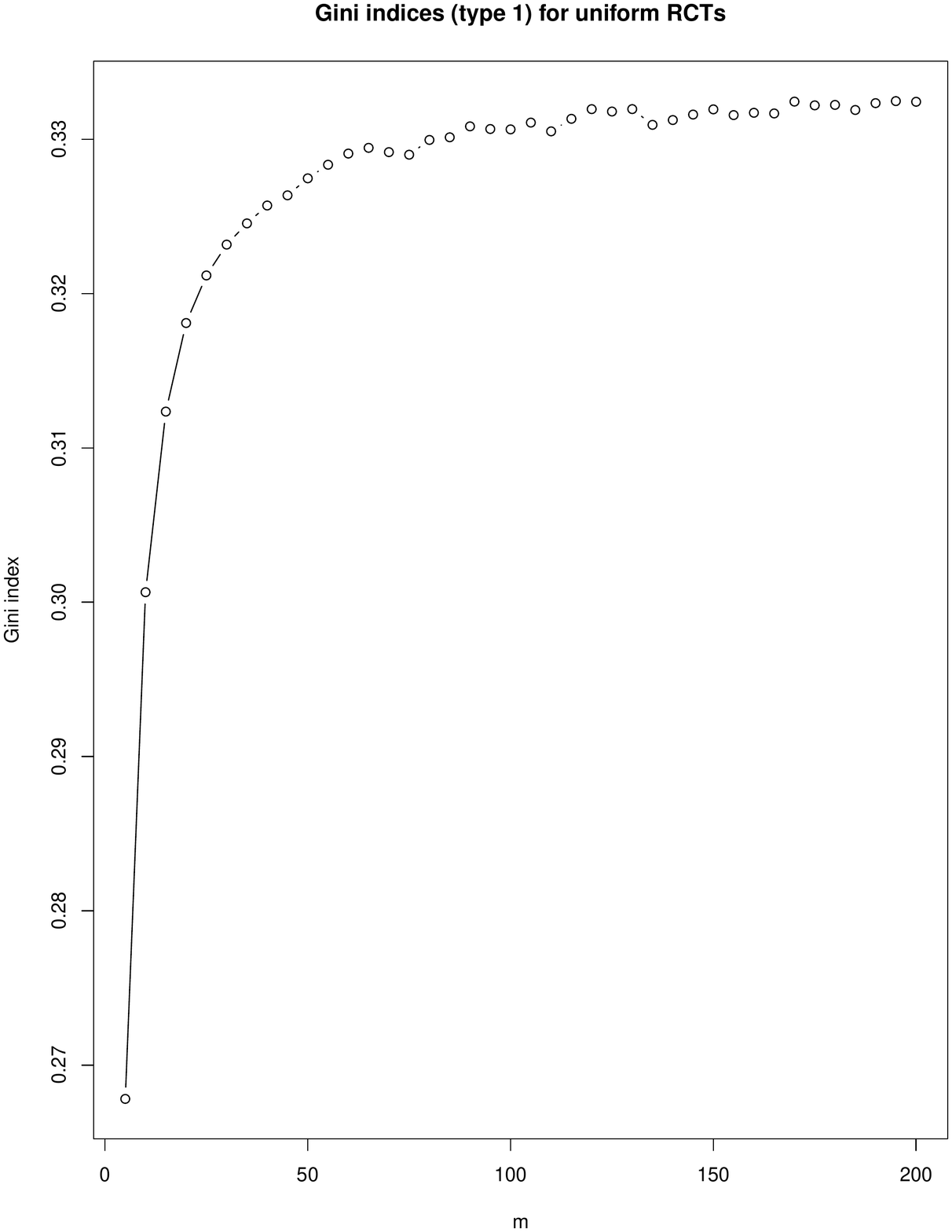}
			\caption{Simulated type {\rm I} Gini index for uniform RCTs at time $n = 500$; $R = 500$ and $m = 5, 10, \ldots, 200$.}
			\label{Fig:ginioneunifRCT}
		\end{minipage}
		\hfill
		\begin{minipage}[b]{0.45\textwidth}
			\includegraphics[width=\textwidth]{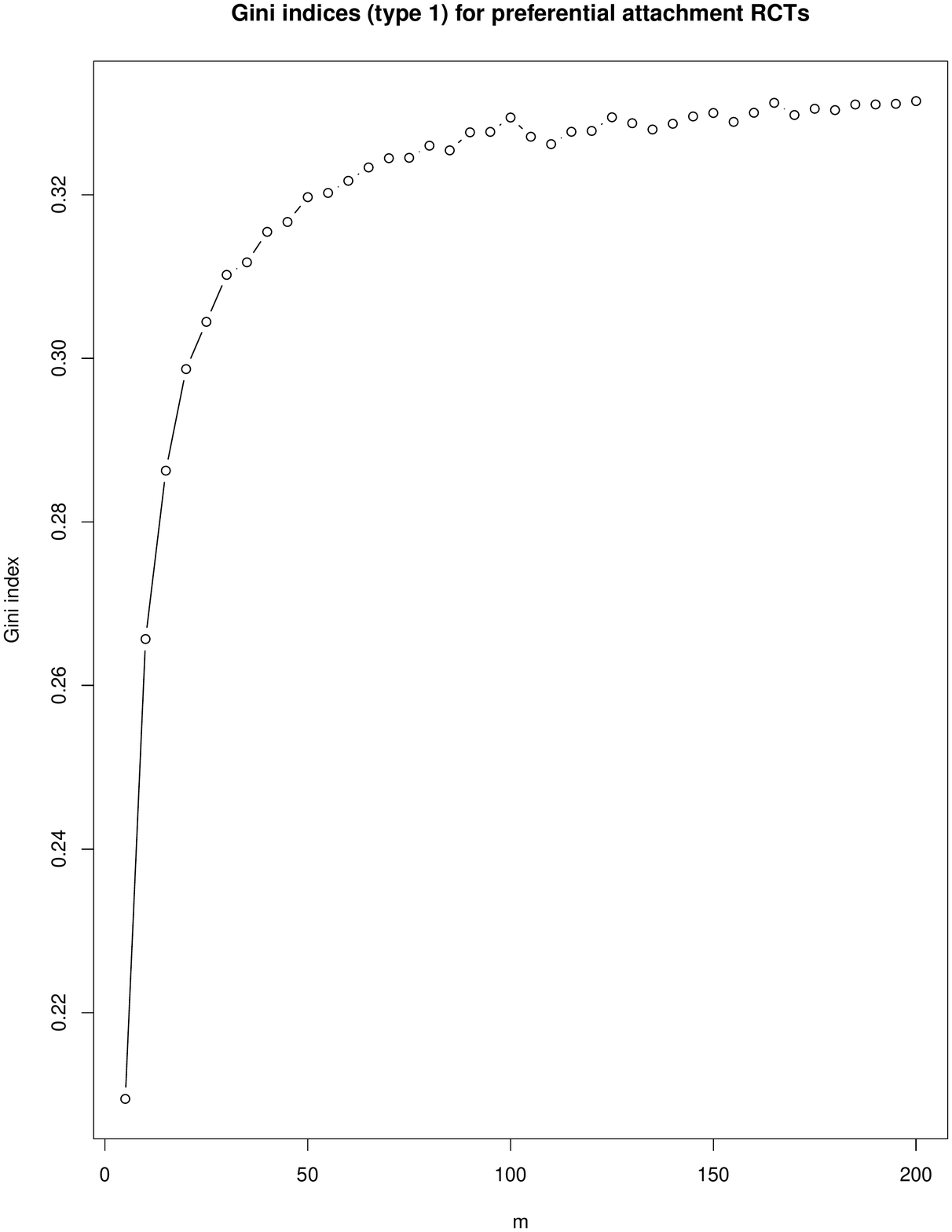}
			\caption{Simulated type {\rm I} Gini index for preferential attachment RCTs at time $n = 500$; $R = 500$ and $m = 5, 10, \ldots, 200$.}
			\label{Fig:ginionepreferRCT}
		\end{minipage}
	\end{center}
\end{figure}

We next conduct an analogous analysis of $\hat{G}_{\rm I}\left(\mathcal{C}^{(P)}\right)$. The analytic result of estimation is presented in Section~\ref{Subsec:gini1p}. We find that $\hat{G}_{\rm I}\left(\mathcal{C}^{(P)}\right)$ approaches $(2m^2 - 7m + 9)/(6m^2 + 3m - 1)$ when $n$ is large. In what follows, $\hat{G}_{\rm I}\left(\mathcal{C}^{(P)}\right)$ also converges to $1/3$ for a large value of $m$. This conclusion is also verified via a numerical experiment with the same parametric setting (as for the uniform case); see Figure~\ref{Fig:ginionepreferRCT}.

According to our computations, we discover that type ${\rm I}$ Gini indices proposed in~\cite{Balaji} are asymptotically identical for two classes of RCTs which grow in completely different manners, suggesting this type of Gini index fails to distinguish the evolutionary behavior and construct feature of the two models. Our conjecture is further verified by four studies of Lorenz curves, depicted in Figure~\ref{Fig:lorenzdist}. We simulate $5000$ uniform RCTs and preferential RCTs at time $n = 5000$ for each of the four values of $m$, which are $5$ (top left), $50$ (top right), $100$ (bottom left) and $500$ (bottom right), respectively. We can see that the Lorenz curves of uniform RCTs and preferential attachment RCTs are close to each other for small values of $m$, but they are indistinguishable for large values of $m$. We thus conclude that type {\rm I} Gini index cannot be used to quantify the inequality of the distribution of the distance measure of the vertices in different classes of RCTs. 
\begin{figure}[H]
	\begin{center}
		\includegraphics[width=\textwidth]{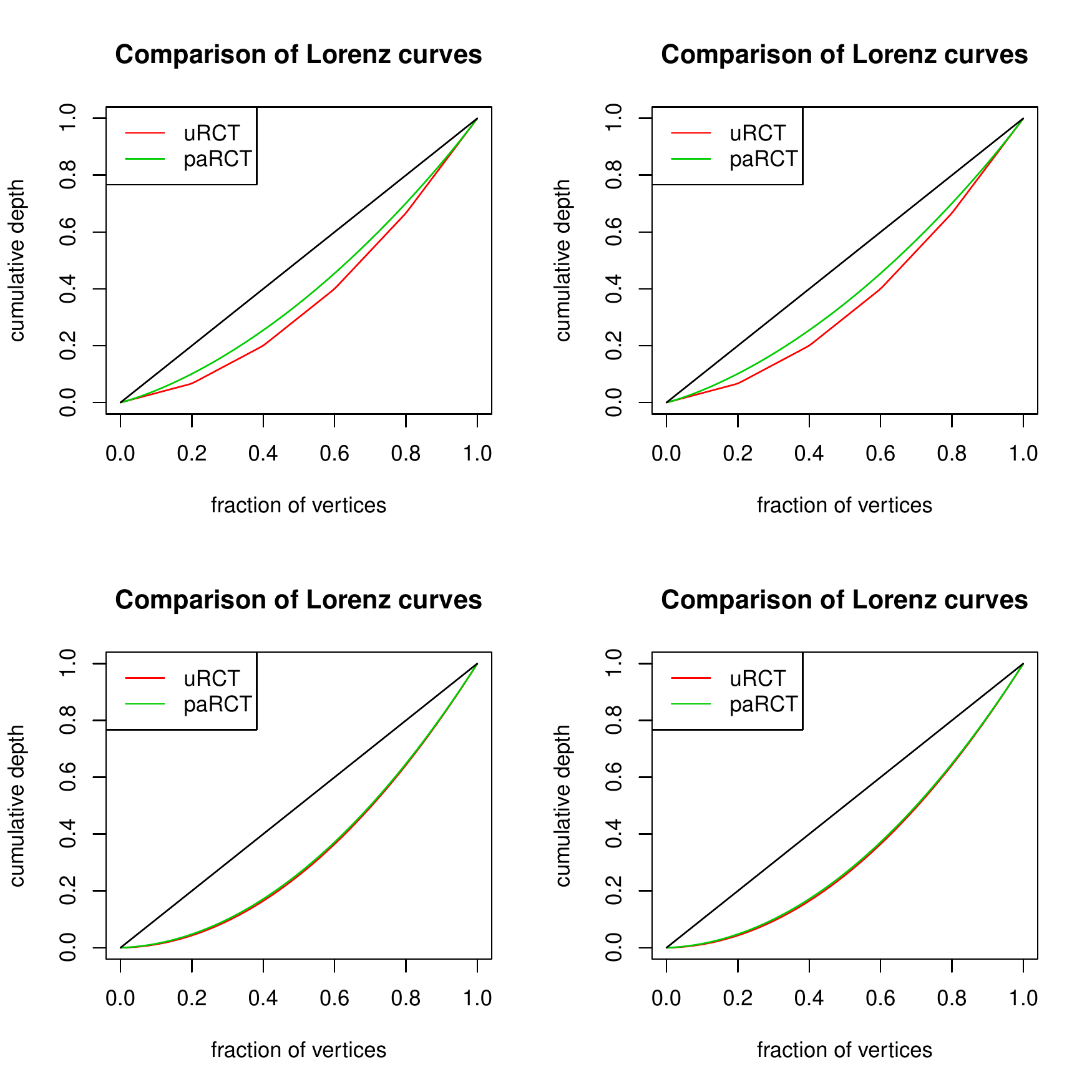}
		\caption{The Lorenz curves (for type {\rm I} Gini index) of uniform RCTs and preferential attachment RCTs for $m = 5, 50, 100, 500$ at time $n = 5000$; $R = 5000$.}
		\label{Fig:lorenzdist}
	\end{center}
\end{figure}

\subsection{Type {\rm II} Gini index}
\label{Subsec:type2}

Alternatively, we propose a new type of Gini index, called type {\rm II} Gini index, which not only accounts for the structure of RCTs, but also precisely characterize the evolution of the RCTs from different classes. 

Instead of including all vertices in our target population, we only consider the $m$ vertices on the central path of a tree at time $n$, i.e., $T_n$. The ``wealth'' of vertex $i$ (for $i = 1, 2, \ldots, m$) is represented by the number of leaves attached to it, i.e., $l_{i, n}$. Thus, we define type {\rm II} Gini index of $T_n$ (at time $n$) which is given by
\begin{equation}
\label{Eq:gini2ext}
G_{\rm II}(T_n) = \frac{\sum_{i = 1}^{m} \sum_{j = 1}^{m} |l_{i, n} - l_{j, n}|}{2 \sum_{j = 1}^{m} \sum_{i = 1}^{m} l_{i, n}}.
\end{equation}

We conduct analogous numerical experiments to calculate type {\rm II} Gini indices for uniform and preferential attachment RCTs. For each class, we simulate $R = 500$ RCTs at time $n = 500$ for $20$ different values of $m$, and compute $G_{\rm II}(T_n)$ for each simulated RCT according to Equation~(\ref{Eq:gini2ext}). Within each class, we take the average of $500$ copies of $G_{\rm II}(T_n)$ as the estimate of type {\rm II} Gini index, and the results of uniform and preferential attachment RCTs are respectively depicted in Figures~\ref{Fig:ginitwounifRCT} and~\ref{Fig:ginitwopreferRCT}.
\begin{figure}[H]
	\begin{center}
		\begin{minipage}[b]{0.45\textwidth}
			\includegraphics[width=\textwidth]{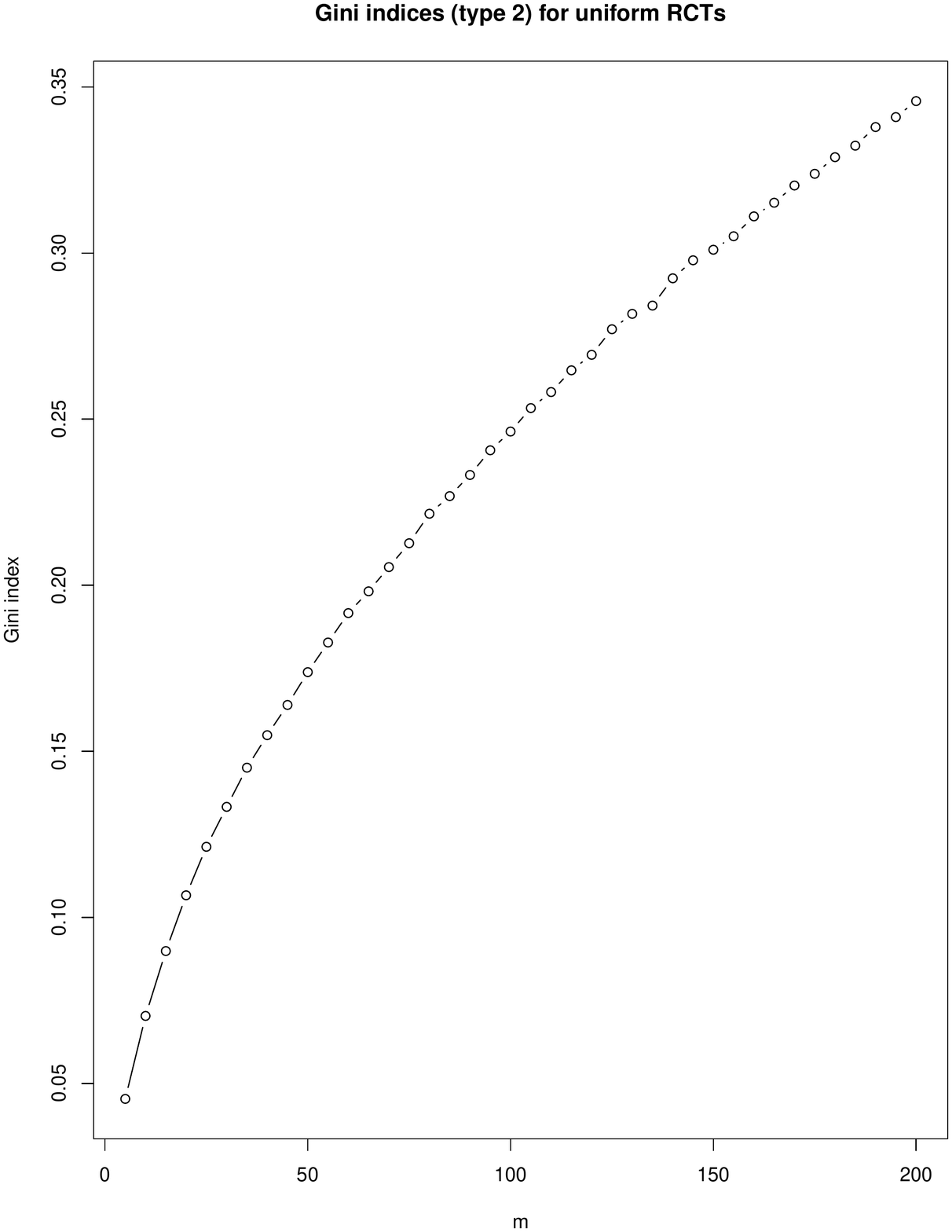}
			\caption{Simulated type {\rm II} Gini index for uniform RCTs at time $n = 500$; $R = 500$ and $m = 5, 10, \ldots, 200$.}
			\label{Fig:ginitwounifRCT}
		\end{minipage}
		\hfill
		\begin{minipage}[b]{0.45\textwidth}
			\includegraphics[width=\textwidth]{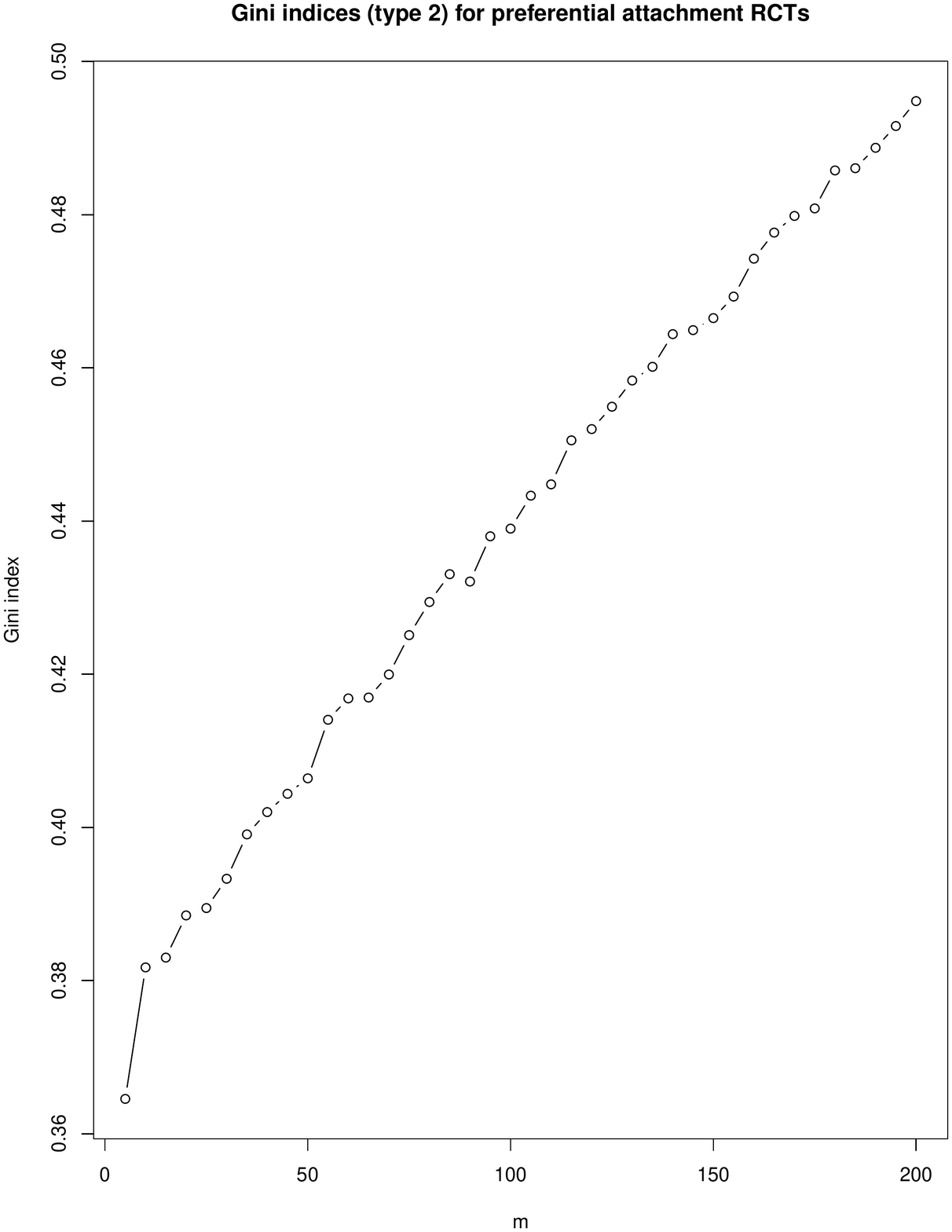}
			\caption{Simulated type {\rm II} Gini index for preferential attachment RCTs at time $n = 500$; $R = 500$ and $m = 5, 10, \ldots, 200$.}
			\label{Fig:ginitwopreferRCT}
		\end{minipage}
	\end{center}
\end{figure}

We discover that type {\rm II} Gini indicies of both classes of RCTs increase with respect to $m$ when time $n$ is large, and the speed of increase of type {\rm II} Gini index of uniform RCTs is much higher than that of preferential attachment RCTs. We make pairwise comparisons and conclude that type {\rm II} Gini index of preferential attachment RCTs is much large than that of uniform RCTs for the same value of $m$, which is also reflected in the Lorenz curves presented in Figure~\ref{Fig:lorenz}.

\begin{figure}[H]
	\begin{center}
		\includegraphics[width=\textwidth]{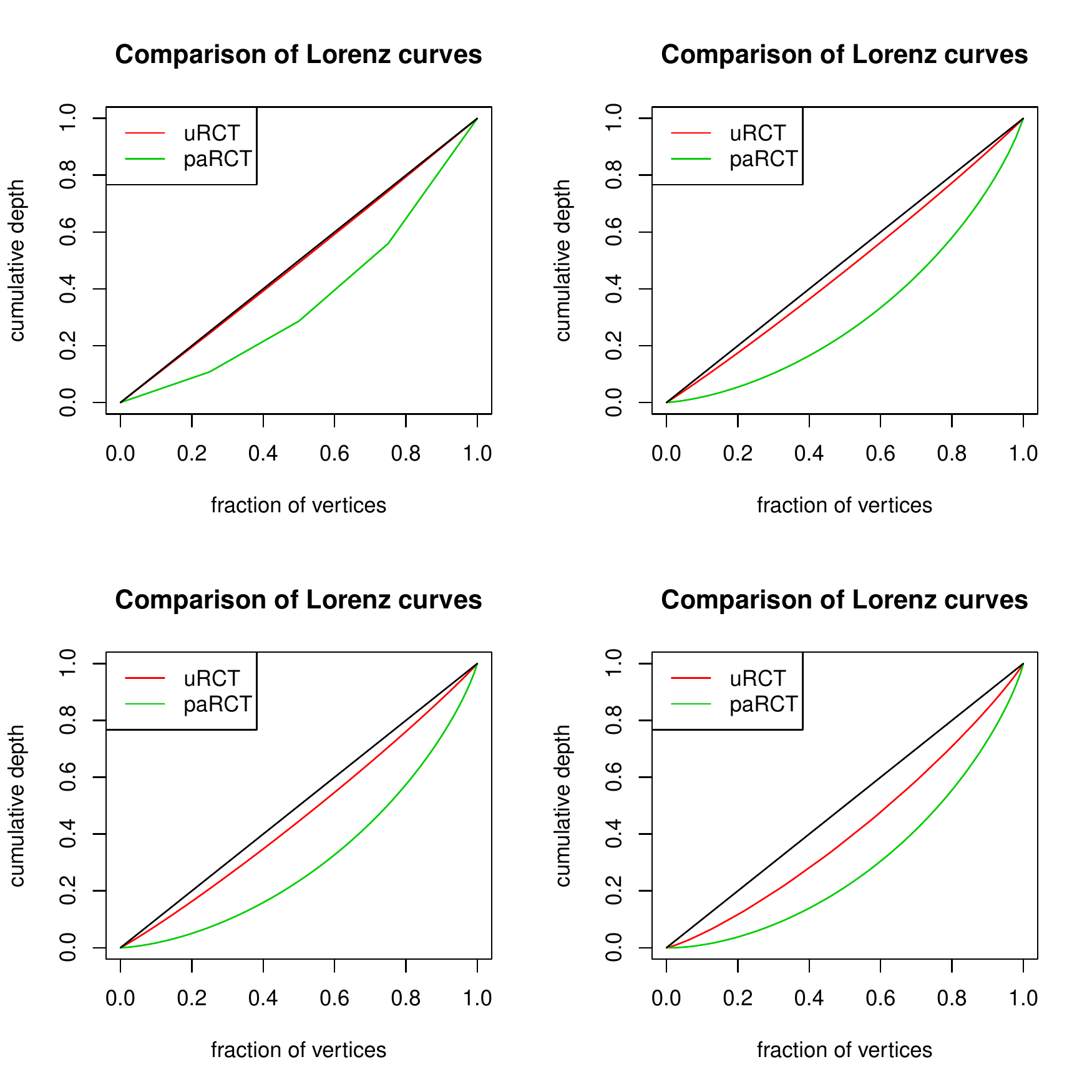}
		\caption{The Lorenz curves (for type {\rm II} Gini index) of uniform RCTs and preferential attachment RCTs for $m = 5, 50, 100, 500$ at time $n = 5000$; $R = 5000$.}
		\label{Fig:lorenz}
	\end{center}
\end{figure}

Our numerical results show that type {\rm II} Gini index of uniform RCTs is small in general when $m \ll n$. The phenomenon that type {\rm II} Gini index preferential attachment RCTs is larger than that of uniform RCTs (for fixed $m$) conforms to the evolution of these two classes of RCTs. The leaves are more likely to be evenly distributed among the $m$ vertices on the central path of uniform RCTs, whereas the leaves are inclined to being connected with the vertices with higher degrees in preferential attachment RCTs, which corresponds to the evolution of this class of RCTs. Conversely, the feature of the growth of preferential attachment RCTs strongly suggests inequality of the distribution of leaves, which is also reflected in the larger value of Gini index in our experiment. Therefore, we conclude that the proposed Gini index successfully characterizes the evolutionary behavior and distinguishes the structure of the two classes of RCTs considered in this paper, which makes it preferred to the one proposed in~\cite{Balaji}.

\section{Concluding remarks}
\label{Sec:conc}
To sum up, we study several properties of RCTs in this manuscript. We consider RCTs which grow in both uniform and nonuniform ways. For a special type of nonuniform RCTs---preferential attachment RCTs, we exploit a generalized \Polya-Eggenberger urn model to determine the exact joint distribution of the degree variables, as well as the asymptotic joint distribution. Multicolor \Polya-Eggenberger urns have been well studied. Three versions of bivariate distributions generated from \Polya-Eggenberger urns are discussed in~\cite{Marshall}. The urn model {\rm I} defined in~\cite[Section 3.1]{Marshall} is a special case of our model for $m = 3$. A general result of strong convergence in a multicolor \Polya\ urn model is given in~\cite{Gouet}, in which the asymptotic joint distribution for the proportions of different types of balls is determined. The asymptotic joint distribution in Corollary~\ref{Cor:asymdist} also can be obtained by applying the result in~\cite[Theorem 3.1]{Gouet}. In addition, we would like to point out that we are able to determine the asymptotic marginal distributions for $D_{i, n}/n$ for each $0 \le i \le m$ based on the fundamental property of Dirichlet distribution; that is,
$$\frac{D_{i, n}}{n} \convD {\rm Beta}(\tau_{i, n}, \tau_0 - \tau_{i, n}).$$ 
For a special case of $m = 2$, $D_{1, n}/n$ and $D_{2, n}/n$ both converge to uniform distributions on $(0, 1)$ asymptotically.

For the Gini index proposed in this manuscript, we are able to apply the estimation developed in~\cite{Gastwirth} to get
\begin{equation}
\hat{G}_{\rm II}(\mathcal{T}) = \frac{\sum_{i = 1}^{m} \sum_{j = 1}^{m} \E[|L_{i, n} - L_{j, n}|]}{2 m \sum_{i = 1}^{m} \E[L_{i, n}]} = \frac{\sum_{i = 1}^{m} \sum_{j = 1}^{m} \E[|L_{i, n} - L_{j, n}|]}{2 m n},
\label{Eq:gini2gen}
\end{equation}
where $L_{i, n}$ is the corresponding random variable with realization of $l_{i, n}$, for $i = 1, 2, \ldots, m$. In our future work, we would like to develop some rigorous estimators for $\E|L_{i, n} - L_{j, n}|$ in Equation~(\ref{Eq:gini2gen}). This is feasible for the uniform case, as $L_{i, n}$'s jointly follow a multinomial distribution. It is well known that an $m$-dimensional multinomial distribution can be approximated by an $(m - 1)$-dimensional multivariate normal. On the other hand, the exact joint distribution of $L_{i, n}$'s is not tractable for the preferential attachment model. Note that the asymptotic joint distribution of $L_{i, n}$'s can be determined by Corollary~\ref{Cor:asymdist}. One possible approach is to consider gamma representations of Dirichlet random variables and establish approximations from there. Further investigations will be conducted and the results will be presented elsewhere.

\section{Appendix}
\label{Sec:app}
\subsection{Type {\rm I} Gini index of uniform RCTs}
\label{Subsec:gini1u}
To compute $\E|d_u - d_v|$ and $\E[d_v]$ in Equation~(\ref{Eq:gini1gen}), we convert them into equivalent expressions in terms of $\E[L_{j, n}L_{i, n}]$ and $\E[L_{i, n}]$ and use the fact that $(L_{1, n}, L_{2, n}, \ldots, L_{m, n})$ follows a multinomial distribution; see Section~\ref{Subsec:URCT}. The estimate of type {\rm I} Gini index of uniform RCTs is given by
\begin{align*}
\hat{G}_{{\rm I}}\left(\mathcal{C}^{(U)}\right) &= \frac{\sum_{i = 1}^{m}\sum_{j = i}^{m} (j - i) \left(1 +　\E[L_{j, n} L_{i, n}]\right) + \sum_{i = 1}^{m}\sum_{j = i}^{m} (j - i + 1) \E[L_{j, n}]}{(m + n)^2\left[\sum_{i = 1}^{m}i\E\left[L_{i, n}\right] + \sum_{i = 1}^{m} (i - 1)\right]}
\\ &= \frac{(2m^2 - 2)n^2 + (m^3 + 4 m^2 - m + 2)n + 2 m^4 - 2m^2}{(6m^2 + 6m)n^2 + 12 m^3 n + 6 m^4 - 6 m^3}.
\end{align*}
\subsection{Type {\rm I} Gini index of preferential attachment RCTs}
\label{Subsec:gini1p}
We perform a similar computation to estimate type {\rm I} Gini index of preferential attachment RCTs. The moments of $L_{i, n}$ and mixed moments of $L_{j, n}L_{i, n}$ can be easily obtained by the results in Proposition~\ref{Prop:mom}. We thus have
\begin{align*}
\hat{G}_{{\rm I}}\left(\mathcal{C}^{(P)}\right) &= \frac{\sum_{i = 1}^{m}\sum_{j = 1}^{m} |j - i| \left(1 +　\E[L_{j, n} L_{i, n}]\right) + \sum_{i = 1}^{m}\sum_{j = 1}^{m} (|j - i| + 1) \E[L_{j, n}]}{2 (m + n)^2\left[\sum_{i = 1}^{m}i\E\left[L_{i, n}\right] + \sum_{i = 1}^{m} (i - 1)\right]}
\\ &= \frac{1}{6(2m - 1)(m - 1)(n + m)(mn + n + m^2 - m)} 
\\ &\qquad{} \times[2(m - 1)(2m^2 - 7m + 9)n^2 + (4 m^4 - 12 m^3 + 23 m^2 - 9m + 6)n
\\ &\qquad\qquad{} + 2m(m - 1)^2(2m - 1)(m + 1)].
\end{align*}



\begin{thebibliography}{99}
	%
	%
	%
	%
	\bibitem{Allen}
	Allen, S. and O'Donnell, R. (2015). Conditioning and covariance on caterpillars. In proceedings of IEEE Information Theory Workshop, Jerusalem, Israel, 1--5
	\bibitem{Balaji}
	Balaji, H. and Mahmoud, H. (2017). The Gini index of random trees with an application to caterpillars. {\em Journal of Applied Probability}, 54, 701--709
	\bibitem{Barabasi}
	Barab\'{a}si, A.-L. and Albert, R. (1999). Emergence of scaling in random networks. {\em Science}, 286, 509--512
	\bibitem{Doob}
	Doob, J. (1990). Stochastic processes. John Wiley \& Sons, Inc., New York
	\bibitem{Eggenberger}
	Eggenberger F. and P\'{o}lya, G. (1923). {\em \"{U}ber die statistik verketteter vorg\"{a}nge. Zeitschrift f\"{u}r Angewandte Mathematik und Mechanik}, 3, 279--289
	\bibitem{ElBasil}
	El-Basil, S. (1987). Applications of caterpillar trees in chemistry and physics. {\em Journal of Mathematical Chemistry}, 1, 153--174
	\bibitem{Fisz}
	Fisz, M. (1955). The limiting distribution of the multivariate distribution. {\em Studia Mathematica}, 14, 272--275
	\bibitem{Gastwirth}
	Gastwirth, J. (1972). The estimation of the Lorenz Curve and Gini index. {\em The Review of Economics and Statistics}, 54, 306--316
	\bibitem{Goswami}
	Goswami, S., Murthy, C. and Das, A. (2016) Sparsity Measure of a Network Graph: Gini Index. {\tt ArXiv:1612.07074 [cs.DM]}
	\bibitem{Gouet}
	Gouet, R. (1997). Strong convergence of proportions in a multicolor P\'{o}lya urn. {\em Journal of Applied Probability}, 34, 426--435
	\bibitem{Graham}
	Graham, R., Knuth, D. and Patashnik, O. (1989). Concrete mathematics. Addison-Wesley Publishing Company, Reading, MA
	\bibitem{Graczyk}
	Graczyk, P. (2007). Gini coefficient: a new way to express selectivity of kinase inhibitors against a family of kinases. {\em Journal of Medicinal Chemistry}, 15, 5773--5779
	\bibitem{Harary}
	Harary, F. and Schwenk, A. (1973). The number of caterpillars. {\em Discrete Mathematics}, 6, 359--365
	\bibitem{Hall}
	Hall P. and Heyde, C. (1980). Martingale limit theory and its application. Academic Press, Inc., New York
	\bibitem{Hobbs}
	Hobbs, A. (1973). Some Hamiltonian results in powers of graphs. {\em Journal of Research of the National Bureau of Standards, Section B: Mathematics and Mathematical Physics}, 77B, 1--10
	\bibitem{Hu}
	Hu, H.-B. and Wang, X.-F. (2008). Unified index to quantifying heterogeneity of complex networks. {\em Physica A: Statistical Mechanics and its Applications}, 387, 3769--3780
	\bibitem{Jamison}
	Jamison, R., McMorris, F. and Mulder, H. (2003) Graphs with only caterpillars as spanning trees. {\em Discrete Mathematics}, 272, 81--95.
	\bibitem{Kennedy}
	Kennedy, B.P., Kawachi, I., Glass, R. and Prothrow-Stith, D. (2008). Income distribution, socioeconomic status, and self rated health in the United States: multilevel analysis. {\em British Medicine Journal}, 317, 917--921
	\bibitem{Kotz}
	Kotz, S., Mahmoud, H. and Robert, P. (2000). On generalized P\'{o}lya\ urn models. {\em Statistics and Probability Letters}, 49, 163--173
	\bibitem{Lee}
	Lee, W.-C. (1997). Characterizing exposure-disease association in human populations using the Lorenz curve and Gini index. {\em Statistics in Medicine}, 16, 729--739
	\bibitem{Lerman}
	Lerman, R. and Yitzhaki, S. (1984). A note on the calculation and interpretation of the Gini index. {\em Economic Letters}, 15, 363--368
	\bibitem{Mahmoud}
	Mahmoud, H. (2009). P\'{o}lya urn models. CRC Press, Boca Raton, FL
	\bibitem{Marshall}
	Marshall, A. and Olkin, I. (1990). Bivariate distributions generated from P\'{o}lya-Eggenberger urn models. {\em Journal of Multivariate Analysis}, 3, 48--65
	\bibitem{Masse}
	Mass\'{e}, A., de Carufel, J., Goupil, A., Lapointe, M., Nadeau, \'{E}. and Vandomme, \'{E}. (2017) Leaf realization problem, caterpillar graphs and prefix normal words. {\tt ArXiv:1712.01942 [math.CO]}
	\bibitem{Miller}
	Miller, Z. (1981). The bandwidth of caterpillar graphs. In proceedings of The 12th Southeastern Conference on Combinatorics, Graph Theory and Computing. Baton Rouge, LA, {\em Congressus Numerantium}, 33, 235--252
	\bibitem{Musiela}
	Musiela, M. and Rutkowski, M. (2005). Martingale methods in financial modelling. Springer-Verlag, Berlin
	\bibitem{Ogwang}
	Ogwang, T. (2000). A convenient method of computing the Gini Index and its standard error. {\em Oxford Bulletin of Economics and Statistics}, 62, 123--129
	\bibitem{Raychaudhuri}
	Raychaudhuri, A. (1995) The total interval number of a tree and the Hamiltonian completion number of its line graph. {\em Information Processing Letters}, 56, 299--306. 
	\bibitem{Roe}
	Roe, B., Yang, H.-J., Zhu, J., Liu, Y., Stancu, I. and McGregor, G. (2005). Boosted decision trees as an alternative to artificial neural networks for particle identification. {\em Nuclear Instruments and Methods in Physics Research Section A: Accelerators, Spectrometers, Detectors and Associated Equipment}, 543, 577--584
	\bibitem{Yule}
	Yule, G. (1925). A mathematical theory of evolution, based on the conclusions of Dr. J.\ C.\ Willis, F.R.S. {\em Philosophical Transactions of the Royal Society B}, 213, 21--87
\end{thebibliography}

\end{document}